\documentclass[11pt]{amsart}
\usepackage{amsmath,amsthm,amsfonts,amssymb,latexsym,enumerate, color}
\usepackage{hyperref}
\usepackage{url}
\usepackage{booktabs}

\newtheorem{thm}{Theorem}[section]

\newtheorem{lem}[thm]{Lemma}

\newtheorem{rem}[thm]{Remark}
\numberwithin{equation}{section}

\begin{document}

\theoremstyle{plain}

\newcommand{\Maxn}{\operatorname{Max_{\textbf{N}}}}
\newcommand{\Syl}{\operatorname{Syl}}
\newcommand{\dl}{\operatorname{dl}}
\newcommand{\Con}{\operatorname{Con}}
\newcommand{\cl}{\operatorname{cl}}
\newcommand{\Stab}{\operatorname{Stab}}
\newcommand{\Aut}{\operatorname{Aut}}
\newcommand{\Out}{\operatorname{Out}}
\newcommand{\Ker}{\operatorname{Ker}}
\newcommand{\fl}{\operatorname{fl}}
\newcommand{\Irr}{\operatorname{Irr}}
\newcommand{\SL}{\operatorname{SL}}
\newcommand{\FF}[1]{\mathbb{F}_{#1}}
\newcommand{\NN}{\mathbb{N}}
\newcommand{\N}{\mathbf{N}}
\newcommand{\C}{\mathbf{C}}
\newcommand{\OO}{\mathbf{O}}
\newcommand{\F}{\mathbf{F}}

\renewcommand{\labelenumi}{\upshape (\roman{enumi})}

\newcommand{\PSL}{\operatorname{PSL}}
\newcommand{\PSU}{\operatorname{PSU}}
\newcommand{\SU}{\operatorname{SU}}
\newcommand{\GU}{\operatorname{GU}}

\providecommand{\V}{\mathrm{V}}
\providecommand{\E}{\mathrm{E}}
\providecommand{\ir}{\mathrm{Irr_{rv}}}
\providecommand{\Irrr}{\mathrm{Irr_{rv}}}
\providecommand{\re}{\mathrm{Re}}

\def\irrp#1{{\rm Irr}_{p'}(#1)}
\newtheorem*{thmA}{Theorem A}
\newtheorem*{corB}{Corollary B}
\newtheorem*{thmC}{Theorem C}

\def\Z{{\mathbb Z}}
\def\C{{\mathbb C}}
\def\Q{{\mathbb Q}}
\def\irr#1{{\rm Irr}(#1)}
\def\irrv#1{{\rm Irr}_{\rm rv}(#1)}
\def \c#1{{\cal #1}}
\def\cent#1#2{{\bf C}_{#1}(#2)}
\def\syl#1#2{{\rm Syl}_#1(#2)}
\def\nor{\triangleleft\,}
\def\oh#1#2{{\bf O}_{#1}(#2)}
\def\Oh#1#2{{\bf O}^{#1}(#2)}
\def\zent#1{{\bf Z}(#1)}
\def\det#1{{\rm det}(#1)}
\def\ker#1{{\rm ker}(#1)}
\def\norm#1#2{{\bf N}_{#1}(#2)}
\def\alt#1{{\rm Alt}(#1)}
\def\iitem#1{\goodbreak\par\noindent{\bf #1}}
   \def \mod#1{\, {\rm mod} \, #1 \, }
\def\sbs{\subseteq}

\newcommand{\fS}{{\mathfrak{S}}}
\def\gc{{\bf GC}}
\def\ngc{{non-{\bf GC}}}
\def\ngcs{{non-{\bf GC}$^*$}}
\newcommand{\notd}{{\!\not{|}}}

\newcommand{\MC}{{\mathcal M}}
\newcommand{\TC}{{\mathcal T}}
\newcommand{\LC}{{\mathcal L}}
\newcommand{\CL}{{\mathcal C}}
\newcommand{\EC}{{\mathcal E}}
\newcommand{\GCD}{\GC^{*}}
\newcommand{\TCD}{\TC^{*}}
\newcommand{\tV}{\tilde{V}}
\newcommand{\bFq}{{\bar{\FF}_q}}
\newcommand{\FD}{F^{*}}
\newcommand{\GD}{G^{*}}
\newcommand{\HD}{H^{*}}
\newcommand{\GCF}{\GC^{F}}
\newcommand{\TCF}{\TC^{F}}
\newcommand{\PCF}{\PC^{F}}
\newcommand{\GCDF}{(\GC^{*})^{F^{*}}}
\newcommand{\RGTT}{R^{\GC}_{\TC}(\theta)}
\newcommand{\RGTA}{R^{\GC}_{\TC}(1)}
\newcommand{\SR}{{^*R}}
\newcommand{\Om}{\Omega}
\newcommand{\eps}{\epsilon}
\newcommand{\varep}{\varepsilon}
\newcommand{\al}{\alpha}
\newcommand{\chis}{\chi_{s}}
\newcommand{\sigmad}{\sigma^{*}}
\newcommand{\PA}{\boldsymbol{\alpha}}
\newcommand{\gam}{\gamma}
\newcommand{\lam}{\lambda}
\newcommand{\la}{\langle}
\newcommand{\ra}{\rangle}
\newcommand{\hs}{\hat{s}}
\newcommand{\htt}{\hat{t}}
\newcommand{\hP}{\hat{P}}
\newcommand{\hR}{\hat{R}}
\newcommand{\hG}{\hat{G}}
\newcommand{\bM}{\bar{M}}
\newcommand{\tv}{\tilde\varepsilon}
\newcommand{\tCk}{\tilde{C}_\kappa}
\newcommand{\orm}{{\mathrm {o}}}
\newcommand{\reg}{{\bf{reg}}}
\newcommand{\tn}{\hspace{0.5mm}^{t}\hspace*{-0.2mm}}
\newcommand{\ta}{\hspace{0.5mm}^{2}\hspace*{-0.2mm}}
\newcommand{\tb}{\hspace{0.5mm}^{3}\hspace*{-0.2mm}}
\def\skipa{\vspace{-1.5mm} & \vspace{-1.5mm} & \vspace{-1.5mm}\\}
\newcommand{\tw}[1]{{}^#1\!}
\renewcommand{\mod}{\bmod \,}

\marginparsep-0.5cm

\renewcommand{\thefootnote}{\fnsymbol{footnote}}
\footnotesep6.5pt

\newcounter{step}
\newcommand{\Step}{\refstepcounter{step} {\sl Step \thestep.}}

\newcounter{little}
\newcommand{\Little}{\refstepcounter{little} {\sl Step \thelittle.}}

\newcommand{\Fit}{\mathbf{F}(G)}

\title[Irreducible Induction and Nilpotent subgroups in finite groups]
{Irreducible Induction and Nilpotent subgroups in finite groups}

\author{Zolt\'an Halasi} \address{Department of Algebra and Number
  Theory, E\"otv\"os University, P\'azm\'any P\'eter S\'et\'any 1/c,
  H-1117, Budapest, Hungary \and Alfr\'ed R\'enyi Institute of
  Mathematics, Hungarian Academy of Sciences, Re\'altanoda utca 13-15,
  H-1053, Budapest, Hungary
  \newline
  ORCID: \url{https://orcid.org/0000-0002-1305-5380}
  } 
\email{zhalasi@cs.elte.hu,  halasi.zoltan@renyi.mta.hu}

\author{Attila Mar\'oti}
\address{Alfr\'ed R\'enyi Institute of Mathematics, Hungarian Academy of
 Sciences, Re\'altanoda utca 13-15, H-1053, Budapest, Hungary}
\email{maroti.attila@renyi.mta.hu}

\author{Gabriel Navarro}
\address{Departament d'\`Algebra, Universitat de Val\`encia, 46100 Burjassot,
Val\`encia, Spain}
\email{gabriel.navarro@uv.es}

\author{Pham Huu Tiep}
\address{Department of Mathematics, Rutgers University, Piscataway, 
NJ 08854, USA}
\email{tiep@math.rutgers.edu}

\date{\today}


\thanks{ The work of the first and second authors on the project
  leading to this application has received funding from the European
  Research Council (ERC) under the European Union's Horizon 2020
  research and innovation programme (grant agreement
  No. 741420). Their work was partly supported by the National
  Research, Development and Innovation Office (NKFIH) Grant
  No.~K115799. They were also supported by the J\'anos Bolyai Research
  Scholarship of the Hungarian Academy of Sciences.}
\thanks{The research of the third author is supported by the Prometeo/Generalitat 
Valenciana, Proyecto MTM2016-76196-P   and FEDER funds.}  
\thanks{The fourth author gratefully acknowledges the support of the NSF (grant DMS-1840702).}  
\thanks{The authors thank E. Vdovin for a helpful clarification on results of \cite{Vdovin}.}

\begin{abstract}
Suppose that $G$ is a finite group and $H$
is a nilpotent subgroup of $G$. If
a character of $H$ induces an irreducible
character of $G$, then the generalized Fitting subgroup
of $G$ is nilpotent.
 \end{abstract}
\maketitle 

\section{Introduction}
Brauer's famous induction theorem
asserts that every irreducible character of $G$
is an integer linear combination of characters induced from
nilpotent subgroups of $G$.  When an irreducible
character is induced from a character of a single nilpotent subgroup of $G$
is a problem that has not been treated until now.

If $\gamma$ is a character of $H$, a subgroup of a finite group $G$,
it is not clear at all when to expect the induced character $\gamma^G$
to be irreducible.  The only case which is understood, using the
Clifford correspondence, is when $H$ happens to contain the stabilizer
of an irreducible character of a normal subgroup $N$ of $G$. In this
case, $N\cent GN \sbs H$, and in a well-defined sense, $H$ is
considered to be a \textit{large} subgroup of $G$: the centralizer of
the core of $H$ in $G$ is contained in $H$.  But of course,
irreducible induction of characters also occurs, we might say by {\sl
  accident}, in other cases.  Even more, some simple groups have
irreducible characters that are induced from linear characters of {\sl
  very easy} subgroups, which of course are core-free.  For instance,
$G=\PSL(2,p)$ with $p\equiv 3 \pmod 4$, has an irreducible character of
degree $p+1$ which is induced from a linear character of the
normalizer $H$ of a Sylow $p$-subgroup of $G$. Here, $H$ is the
semidirect product of the cyclic group of order $p$ by the cyclic
group of order $(p-1)/2$.  The key thing is that it does not matter
how easy $H$ is as long as it is not nilpotent.
  
\medskip
We write ${\bf F^*}(G)$ for the generalized Fitting subgroup of $G$, and
recall its fundamental property that $\cent G{{\bf F}^*(G)} \sbs {\bf F^*}(G)$.
Also, ${\bf F}(G)$ is the Fitting subgroup of $G$.

\begin{thmA}\label{thm:A}
Let $G$ be a finite group and let $H$ be a maximal nilpotent subgroup
of $G$. Suppose that $\gamma \in \irr H$ is such that $\gamma^G \in \irr G$.
Then ${\bf F^*}(G) \sbs H$. In particular ${\bf F^*}(G)={\bf F}(G)$.
\end{thmA}

Irreducible characters that are induced from Sylow subgroups were
studied in \cite{RS}. Their main result is easily seen to be a consequence
of Theorem A.

\medskip

\begin{corB}
If $G$ is a finite group, $P \in \syl pG$ and $\gamma \in \irr P$ induces irreducibly 
to $G$, then $\zent P \nor\nor G$.
\end{corB}

\begin{proof}
Suppose that $P$ is contained in a nilpotent subgroup $H=P\times K$ of $G$.
Since $\gamma^{H}$ is irreducible, $\cent HP\sbs P$ by elementary
character theory, and thus $H=P$.  Therefore ${\bf F}^*(G)=\oh pG$ by Theorem A.
Now $\zent P \sbs \cent G{{\bf F}^*(G)} \sbs {\bf F}^*(G)$, and we are done.
\end{proof}

We use a combination of several techniques to prove Theorem A.
One of them is to prove that, in general, nilpotent subgroups
are small in almost-simple groups. (This complements work
in \cite{Vdovin}.) Some delicate character theory reductions are needed to bring
this fact into the proof.  The cases in which this does not happen are dealt
using character theory of certain quasisimple groups.
Perhaps it is worth to state here what we shall 
need and prove below.

\medskip

\begin{thmC}
  If $Y$ is a nilpotent subgroup in an almost simple group $X$, then
  ${|Y|}^{2} < |X|$.
\end{thmC}

\section{Almost simple groups}

We begin with the proof of Theorem C.  Note that during the proof we
try to show the slightly stronger inequality $2|Y|^2\leq |X|$. It
turns out that in most cases even this stronger statement holds.
Then we will be able to provide a very short list of groups where
the inequality $2|Y|^2\leq |X|$ fails (see Theorem \ref{almostsimple}).

Let $X$ be an almost simple group with socle $S$ and let $m(S)$ denote
the largest possible size of a nilpotent subgroup in $S$. Let $Y$ be a
nilpotent subgroup in $X$.

\medskip
\Little~~{\sl If $S \cong \mathrm{Alt}(k)$, then $2 \cdot {|Y|}^{2} \leq |X|$, unless $k \in \{ 5,6 \}$ when ${|Y|}^{2} < |X|$.}\label{Step1}
\medskip

Assume first that $k \geq 9$. In this case we have $2 \cdot {|Y|}^{2} \leq 2^{2k-1} < k!/2 \leq |X|$ by a result from \cite{Dixon} stating that a nilpotent permutation group of degree $k$ has size at most $2^{k-1}$.

By using \cite[Theorem 2.1]{Vdovin}, the following table 
contains the value of $m(S)$ and $|\Out(S)|$ when $S=\alt k$ for $k=5,6,7,8$. 
\[
\begin{array}{@{}lcccc@{}}
\toprule
k=          &5  &6  &7  &8  \\
\midrule
m(\alt k)   &5  &9  &12 &2^6\\
\Out(\alt k)&2  &4  &2  &2  \\
\bottomrule
\end{array}
\]
Now, $|Y|\leq m(S)\cdot |X:S|< \sqrt{|X|}$ holds for $k=5,6$, while
$|Y|\leq m(S)\cdot |X:S|< \sqrt{|X|/2}$ holds for $k=7,8$.

\medskip
\Little~~{\sl If $S$ is a sporadic simple group or the Tits group, then $2 \cdot {|Y|}^{2} \leq |X|$.} \label{Step2}
\medskip

Let $S$ be a sporadic simple group. The largest possible size of a nilpotent subgroup in $S$ is equal to the size of a Sylow subgroup of $S$ by 
\cite[Section 2.4]{Vdovin}. This in turn is less than ${|S|}^{1/2}/2$ by \cite{ATLAS}. Since $|\mathrm{Out}(S)| \leq 2$, the claim follows.

Let $S$ be the Tits group. If $X = S$, then $|Y|$ is at most the size of a Sylow subgroup in $X$, by the proof of \cite[Theorem 2.2]{Vdovin}, and so $|Y| \leq 2^{11}$ by \cite{ATLAS}. If $X = \mathrm{Aut}(S)$, then $|Y| \leq 2^{12}$, since the outer automorphism group of $S$ has size $2$, and again, the claim follows.  

\medskip

From now on, assume that $S$ is a finite simple group of Lie type different from $\mathrm{Alt}(5)$ and $\mathrm{Alt}(6)$.  Note that $2 \cdot {|Y|}^{2} \leq |X|$ whenever $2 \cdot |\mathrm{Out}(S)| \cdot {m(S)}^{2} \leq |S|$.

\medskip
\Little~~{\sl If $m(S) > {|S|}_p$ where $p$ is any natural characteristic for $S$, then $2 \cdot {|Y|}^{2} \leq |X|$.} \label{Step3}
\medskip

There are three possibilities for $S$ according to \cite[Table 3]{Vdovin}: (1) $S \cong \mathrm{PSL}(2,2^{m})$ for some $m \geq 3$; (2) $S \cong \mathrm{PSL}(2,2^{m}+1)$ for some $m \geq 4$; and (3) $S \cong \mathrm{PSU}(3,3)$. (Note that the third group in \cite[Table 3]{Vdovin} is solvable.) 

Consider cases (1) and (2). Let $q$ be the size of the field over which $S$ is defined. Write $d$ to satisfy $q = p^{d}$. Then 
$\mathrm{Out}(S) \cong C_{(2,q-1)} \times C_{d}$ where $(2,q-1)$ is the greatest common divisor of $2$ and $q-1$, and $m(S) \leq q+1$ by 
\cite[Table 3]{Vdovin}. It is straightforward to check that $2 \cdot {(2,q-1)}^{2} \cdot d \cdot (q+1) \leq q(q-1)$, establishing  
$2 \cdot |\mathrm{Out}(S)| \cdot {m(S)}^{2} \leq |S|$.

In case $S \cong \mathrm{PSU}(3,3)$ we have $m(S) = 32$, $|S| = 6048$, and $|\mathrm{Out}(S)| = 2$ by \cite{ATLAS}. Thus $2 \cdot |\mathrm{Out}(S)| \cdot {m(S)}^{2} \leq |S|$. 

\medskip

We may now assume that $S$ is a finite simple group of Lie type of Lie rank $\ell$ defined over a field of size $q$ in characteristic $p$ and $m(S) = {|S|}_{p}$. 

\medskip
\Little~~{\sl $2 \cdot {|Y|}^{2} \leq |X|$ unless possibly if $\ell = 1$ and $q < 2^{12}$, or $2 \leq \ell \leq 9$ and $q < 2^{6}$.} \label{Step4}
\medskip

By the order formulas for $|S|$ and ${|S|}_{p}$ (see \cite[page 170]{KL}) and by $\prod_{i=1}^{\infty} (1 - 2^{-i}) > 2/7$ (see the proof of \cite[Lemma 3.2]{BEGHM}), we have $$2 \cdot \min \{ \ell + 1, q+1 \} \cdot \frac{|S|}{{(|S|_{p})}^{2}} > \frac{7}{8} \cdot \prod_{i=1}^{\infty} (1 - 2^{-i}) \cdot q^{\ell} > \frac{1}{4} \cdot q^{\ell}.$$ Again by \cite[page 170]{KL}, we have $2 \cdot |\mathrm{Out}(S)| \leq 8 \cdot \min \{ \ell +1, q+3 \}  \cdot \log_{p} q$ unless $\ell = 4$, when $2 \cdot |\mathrm{Out}(S)| \leq 48 \cdot \log_{p}q$. These are smaller than $$\frac{1}{8 \cdot \min \{ \ell +1, q+1 \}} \cdot q^{\ell} < \frac{|S|}{{(|S|_{p})}^{2}},$$ unless $\ell = 1$ and $q < 2^{12}$, or $2 \leq \ell \leq 9$ and $q < 2^{6}$. 

\medskip
\Little~~{\sl If $S$ is not isomorphic to any of the groups $\mathrm{Alt}(5)$, $\mathrm{Alt}(6)$, $\mathrm{PSL}(2,7)$, $\mathrm{PSL}(3,4)$, $\mathrm{PSp}(4,3)$, $\mathrm{PSU}(4,3)$, then $2 \cdot {|Y|}^{2} \leq |X|$.} \label{Step5}
\medskip

By a Gap \cite{GAP} computation (using Step \ref{Step4}) together with \cite[page 170]{KL} we get that $2 \cdot |\mathrm{Out}(S)| \cdot {(|S|_{p})}^{2} < |S|$ unless $S$ is isomorphic to $\mathrm{Alt}(5)$, $\mathrm{Alt}(6)$, $\mathrm{PSL}(2,7)$, $\mathrm{PSL}(3,4)$, $\mathrm{PSp}(4,3)$, $\mathrm{PSU}(3,5)$, $\mathrm{PSU}(3,8)$, $\mathrm{PSU}(4,3)$, $\mathrm{PSU}(6,2)$, $\mathrm{P\Omega}^{+}(8,2)$, or $\mathrm{P\Omega}^{+}(8,3)$.
If $S \cong \mathrm{PSU}(3,5)$, $\mathrm{PSU}(3,8)$, $\mathrm{PSU}(6,2)$, $\mathrm{P\Omega}^{+}(8,2)$, or $\mathrm{P\Omega}^{+}(8,3)$, then a computation shows that $2 \cdot {|Y|}^{2} \leq |X|$, using the fact that the outer automorphism group of $S$ is not nilpotent. 

\medskip
 {\sl Final Step.} 
\medskip

We have $|\mathrm{Out}(S)| \cdot {(|S|_{p})}^{2} < |S|$ unless $S$ is isomorphic to $\mathrm{PSL}(3,4)$, $\mathrm{PSU}(4,3)$, or $\mathrm{P\Omega}^{+}(8,3)$. In the latter case $2 \cdot {|Y|}^{2} \leq |X|$ by Step \ref{Step5}.

Let $S \cong \mathrm{PSL}(3,4)$. The outer automorphism group of $S$ is $C_2 \times \mathrm{Sym}(3)$, a non-nilpotent group. We have ${|Y|}^{2} < |X|$ unless possibly if $Y$ projects onto $X/S$ and $X/S$ is cyclic of order $6$. In this exceptional case every element of order $6$ in $X$ has centralizer of order at most $54$ by \cite{ATLAS} and so $|Y| \leq 54$ giving ${|Y|}^{2} < |X|$. 

Finally, we need to check that $|Y|^2<|X|$ holds for the case $S \cong
\PSU(4,3)$. Using information from $\mathrm{Out}(S)$, the
order of $S$, the fact that $m(S) = 3^{6}$, and that the sizes of the
centralizers of elements of orders $10$ or $14$ in $\mathrm{Aut}(S)$
are too small (at most $56$) by \cite{ATLAS}, one can prove that
$|Y|^2<|X|$ except possibly if $X=\mathrm{Aut}(\PSU(4,3))$, $YS=X$ and
the set of prime divisors of $|Y|$ is $\{2,3\}$.  Let us assume that
this is the case. Then the nilpotency of $Y$ guarantees that $Y$
cannot contain a maximal unipotent subgroup of $S$. Therefore, if the
Sylow $2$-subgroup of $Y$ is disjoint from $S$, then $|Y|^2\leq
(8\cdot 3^{5})^2<8\cdot |\PSU(4,3)|=|X|$. Otherwise, let $Z=\{\alpha\in
\FF 9^\times\,|\,\alpha^4=1\}=\zent{\SU(4,3)}$ and $Z<K<\SU(4,3)$ such that
$|K:Z|=2$ and $K/Z$ is normal in $Y\cap \PSU(4,3)$.  Let $V$ be a 4 dimensional
non-degenerate Hermitian space over $\FF 9^\times$ with Hermitian
product $(\;,\;)$ and identify $\SU(4,3)$ with the special unitary group
on $V$ preserving $(\;,\;)$. Let $x\in K\setminus Z$. Then
$x$ is diagonalisable with respect to a suitable basis
of $V$.  Since $x^2$ is a scalar transformation, all
the eigenvalues of $x$ are $\pm \gamma$ for some $\gamma\in \FF
9^\times$. Let $V_1$ and $V_2$ be the eigenspaces corresponding to
$\gamma$ and $-\gamma$, respectively.  We may assume that
$\dim(V_1)\geq \dim(V_2)$, so either $\dim(V_1)=\dim(V_2)=2$ or
$\dim(V_1)=3$ and $\dim(V_2)=1$. If $u$ is a non-singular eigenvector of
$x$, then $0\neq (u,u)=(x(u),x(u))=\gamma^{3+1}(u,u)$, so
$\gamma^4=1$. In that case $\det{x}=1$ holds only if 
$\dim(V_1)=\dim(V_2)=2$. Now, if $\dim(V_1)=3$, then there must be 
a non-singular eigenvector of $x$ in $V_1$, which leads to a contradiction. 
Thus, $\dim(V_1)=\dim(V_2)=2$ and $\gamma^4=1$ must hold. 
If $v_1\in V_1$ and $v_2\in V_2$ are arbitrary, then 
$(v_1,v_2)=(x(v_1),x(v_2))=(\gamma\cdot v_1,-\gamma\cdot v_2)=
-\gamma^{3+1}(v_1,v_2)=-(v_1,v_2)$, so $v_1\perp v_2$. Thus, we get that
$V_1$ and $V_2$ are orthogonal complements to each other, so both 
$V_1$ and $V_2$ are non-degenerate subspaces. 
Let $Z<N<\GU(4,3)=\GU(V)$ with $N/Z=Y\cap \GU(4,3)$, so $|Y|=|N|/2$. 
Since $N$ normalises $K$, it permutes the  homogeneous components of $K$ 
by Clifford theory, which are $V_1$ and $V_2$. It follows that 
$N\leq (\GU(V_1)\times \GU(V_2)) \rtimes C_2\simeq \GU(2,3)\wr C_2$.
Using that $N$ is nilpotent, we have $|N|\leq 32^2\cdot 2=2^{11}$, so 
$|Y|^2\leq 2^{20}< 8\cdot |\PSU(4,3)|=|X|$ follows.

\medskip

The following is essentially a consequence of the proof of Theorem C. 

\begin{thm}\label{almostsimple}
Let $Y$ be a nilpotent subgroup in an almost simple group $X$ with
socle $S$. Assume that $Y \cap S$ is a $p$-group for some prime
$p$. Then $2 \cdot {|Y|}^{2} \leq |X|$ except when
$$S \in
\mathrm{LIST} = \{ \mathrm{Alt}(5), \mathrm{Alt}(6), \mathrm{PSL}(2,7), \mathrm{PSL}(3,4),
\mathrm{PSU}(4,3) \}$$ and $Y \cap S$ is a Sylow $p$-subgroup in
$S$. Moreover, if $S \cong \mathrm{Alt}(5)$ or $\mathrm{Alt}(6)$ and the inequality fails,
then $p=2$.
\end{thm}

\begin{proof}
Let $S \cong \mathrm{PSp}(4,3)$. We have $|S| = 2^{6} \cdot 3^{4} \cdot 5$ and $|\mathrm{Out}(S)| = 2$ by \cite{ATLAS}. Using these and the fact that the centralizer of a Sylow $3$-subgroup in $\mathrm{Aut}(\mathrm{PSp}(4,3))$ has size $3$, we get the inequality $2 \cdot {|Y|}^{2} \leq |X|$. If $S \not\in \{ \mathrm{Alt}(5), \mathrm{Alt}(6), \mathrm{PSL}(2,7), \mathrm{PSL}(3,4), \mathrm{PSU}(4,3) \}$, then $2 \cdot {|Y|}^{2} \leq |X|$ by Step \ref{Step5}. 

Now assume that $S \in \{ \mathrm{Alt}(5), \mathrm{Alt}(6), \mathrm{PSL}(2,7), \mathrm{PSL}(3,4), \mathrm{PSU}(4,3) \}$. 

Let $X = \mathrm{Sym}(5)$. If $|Y| \leq 6$, then $2 \cdot {|Y|}^{2} \leq |X|$ follows. Otherwise $Y$ is a Sylow $2$-subgroup of $X$ and $2 \cdot {|Y|}^{2} > |X|$. If $X = \mathrm{Alt}(5)$, then $2 \cdot {|Y|}^{2} \leq |X|$.   

Let $S = \mathrm{Alt}(6)$. If $Y \cap S$ is different from a Sylow $2$-subgroup and different from a Sylow $3$-subgroup of $S$, then $2 \cdot {|Y|}^{2} \leq |X|$. Let $Y \cap S$ be a Sylow $3$-subgroup of $S$. Then $|Y| = 9$ in case $X = \mathrm{Alt}(6)$ and $|Y| \leq 18$ otherwise. We conclude that $2 \cdot {|Y|}^{2} \leq |X|$.  

Finally, assume that $Y \cap S$ is not a Sylow $p$-subgroup of $S$ where $S$ is any of the groups $\mathrm{PSL}(2,7)$, $\mathrm{PSL}(3,4)$, $\mathrm{PSU}(4,3)$. Then $2 \cdot {|Y|}^{2} \leq |X|$ by \cite{ATLAS}. 
\end{proof}

\section{Quasisimple Groups}

The goal of this section is to prove the following.

\begin{thm}\label{qs}
Suppose that $G$ is a quasisimple group, with $S=G/\zent G$ in 
$$\mathrm{LIST} = \{ \mathrm{Alt}(5), \mathrm{Alt}(6), \mathrm{PSL}(2,7), \mathrm{PSL}(3,4), \mathrm{PSU}(4,3) \}.$$
Let $H \geq \zent G$ be a nilpotent subgroup of $G$ such that $H/\zent G$ is a Sylow $p$-subgroup of $G/\zent G$ for some prime $p$. 
If $G/\zent G \cong \mathrm{Alt}(5)$ or $\mathrm{Alt}(6)$, then assume in addition that $p = 2$.
If $\gamma \in \Irr(H)$, then $\gamma^G$ has at least two irreducible constituents with different degrees.  
\end{thm}

We will need the following technical lemma:

\begin{lem}\label{key}
Let $G$ be a finite group and let $H \geq \zent G$ be a nilpotent subgroup of $G$ such that $H/\zent G$ is a Sylow $p$-subgroup of $G/\zent G$ for some prime $p$. Suppose that all irreducible constituents of $\gamma^G$ are of the same degree $D$ for some $\gamma \in \Irr(H)$. Then the following 
statements hold for any $g \in G \smallsetminus \zent G$.
\begin{enumerate}[\rm(i)]
\item $\gamma^G(g) = 0$ if $g \notin \cup_{x \in G}H^x$. In particular, $\gamma^G(g) = 0$ if the coset $g\zent G$ has order coprime to $p$ in $G/\zent G$.
\item Suppose there exist some $\al \in \C$ and an algebraic conjugate $\al^*$ of $\alpha$ such that $\chi(g) \in \{\al,\al^*\}$ for all 
$\chi \in \Irr(G)$ of degree $D$. If $\gamma^G(g) = 0$, then $\al+\al^* = 0$.
\item Suppose $p \nmid |\zent G|$ so that $H = P \times \zent G$ for some $P \in \Syl_p(G)$. Then all irreducible characters of $G$ that lie
above both $\gamma|_P$ and $\gamma|_{\zent G}$ must have the same degree.
\end{enumerate}
\end{lem}

\begin{proof}
(i) and (iii) are obvious. 

For (ii), write $\gamma^G = \sum^k_{i=1}\chi_i$ with $\chi_i \in \Irr(G)$ of degree $D$. By the assumption, $\chi_i(g) = \al$ or $\al^*$. 
Now if $\al=\al^*$, then $0=\gamma^G(g) = k\al$ and so $\al=0$. We may now assume that $\al \neq \al^*$ and that
\begin{equation}\label{eq:qs1}
  \chi_{1}(g) = \ldots = \chi_{j}(g) = \al,~\chi_{j+1}(g) = \ldots = \chi_k(g) = \al^*,
\end{equation}  
for some $1 \leq j \leq k-1$. Since the set of $\chi \in \Irr(G)$ of given degree $D$ is 
stable under the action of $\Gamma := {\rm {Gal}}(\Q_{|G|}/\Q)$, the set $\{\al,\al^*\}$ is $\Gamma$-stable, and  there is some $\sigma \in \Gamma$ that sends $\al$ to $\al^*$, whence $\sigma(\al^*) = \al$. Now, by \eqref{eq:qs1} we have 
$$j\al+(k-j)\al^* = \gamma^G(g) = 0 = \sigma(0) = \sigma(\gamma^G(g)) = j\al^*+(k-j)\al,$$
whence $k(\al+\al^*) = (j\al+(k-j)\al^*)+(j\al^*+(k-j)\al) = 0$ and $\al+\al^* = 0$ as stated.
\end{proof}

\begin{proof}[Proof of Theorem \ref{qs}]
Assume the contrary: all irreducible constituents of $\gamma^G$ have the same degree $D$. Again write 
\begin{equation}\label{eq:qs2}
  \gamma^G = \sum^k_{i=1}\chi_i,
\end{equation}
with $\chi_i \in \Irr(G)$ of degree $D$. Denoting $\lam := \gamma|_{\zent G}$, we see that all $\chi_i$ in \eqref{eq:qs2} lie above $\lam$. 
Modding out the quasisimple group $G$ by $\Ker(\lam)$, we may therefore assume that all $\chi_i$ are faithful characters of $G$.

We will analyze all possible cases for $S = G/\zent G$.
We will use the notation $b5 = (-1+\sqrt{5})/2$ and $b7 = (-1+\sqrt{-7})/2$ of \cite{ATLAS}, and freely use the character tables of $G$ as listed in
\cite{ATLAS}. Also, $\reg_P$ will denote the regular character of $P \in \Syl_p(G)$.


\smallskip
{\it Case 1:} $S = \mathrm{Alt}(5)$. Then we have $p=2$ by hypothesis. Taking $g \in G$ of order $5$, we see that $g$ fulfills the conditions of 
\ref{key}(ii): indeed, 
\begin{equation}\label{eq:qs3}
\al = \left\{ \begin{array}{rl}
       j, & D \equiv j (\mod 5) \mbox{ with }j \in \{0,\pm 1\},\\ 
       b5, &  D \equiv 2 (\mod 5),\\ 
       -b5, & D \equiv 3 (\mod 5).\end{array} \right.
\end{equation}
By Lemma \ref{key} applied to $g$ we have that $\gamma^G(g) = 0$ and $\al+\al^*=0$. As $b5+b5^* = -1$, we conclude that
$5|D$, and so in fact $D=5$ and $\zent G = 1$. In this case, $\chi_i(h) = -1 \neq 0$ for an element $h \in G$ of order $3$, contradicting
Lemma \ref{key} applied to $h$.

\smallskip
{\it Case 2:} $S = \PSL(2,7)$. Taking $g \in G$ of order $7$, we see that $g$ fulfills the conditions of 
\ref{key}(ii) with 
\begin{equation}\label{eq:qs4}
\al = \left\{ \begin{array}{rl}
       j, & D \equiv j (\mod 7) \mbox{ with }j \in \{0,\pm 1\},\\ 
       b7, &  D \equiv 3 (\mod 7),\\ 
       -b7, & D \equiv 4 (\mod 7).\end{array} \right.
\end{equation}
Suppose first that $p \neq 7$. By Lemma \ref{key} applied to $g$ we have that $\gamma^G(g) = 0$ and $\al+\al^*=0$. As $b7+b7^* = -1$, we conclude that
$7|D$, and so in fact $D=7$ and $\zent G = 1$. In this case, $\chi_i(h) = -1 \neq 0$ for an element $h \in G$ of order $2$, contradicting
Lemma \ref{key} applied to $h$.

Assume now that $p=7$. Then $H = P \times \zent G$ with $P \in \Syl_p(G)$. 
If $\zent G = 1$, then there are characters in $\Irr(G)$ of degree $7$ and degree $8$,
each containing $\reg_P$ on restriction to $P$. If $|\zent G| = 2$, then there exist $\theta_i \in \Irr(G)$, $i = 1,2,3$, with 
$\theta_1$ of degree $8$ containing $\reg_P$, $\theta_2$ of degree $6$ containing $\reg_P-1_P$, and 
$\theta_3$ of degree $4$ containing $1_P$. This is impossible by Lemma \ref{key}(iii).

\smallskip
{\it Case 3:} $S = \mathrm{Alt}(6)$. Then we have $p=2$ by hypothesis. Taking $g \in G$ of order $5$, we see that $g$ fulfills the conditions of 
\ref{key}(ii) with $\al$ specified in \eqref{eq:qs3}.
By Lemma \ref{key} applied to $g$ we have that $\gamma^G(g) = 0$ and $\al+\al^*=0$. As $b5+b5^* = -1$, we conclude that $5|D$; 
in particular, $|\zent G| \leq 3$.

Assume $\zent G = 1$, so that $D \in \{5,10\}$. If $D = 10$, then $\chi_i(h) = 1 \neq 0$ for an element $h \in G$ of order $3$, contradicting
Lemma \ref{key} applied to $h$. Suppose $D=5$. Then $\chi_i \in \{\theta,\theta^*\}$ for all $\chi_i$ in \eqref{eq:qs2}, 
where $\theta(1)=\theta^*(1) = 5$ and 
$$(\theta(x),\theta^*(x)) = (2,-1),~~(\theta(y),\theta^*(y)) = (-1,2)$$
for some elements $x,y \in G$ of order $3$. Without loss we may assume that
$$\chi_1 = \ldots = \chi_j = \theta,~\chi_{j+1} = \ldots = \chi_k = \theta^*$$
for some $1 \leq j \leq k$. Applying Lemma \ref{key} to $x$ and to $y$, we obtain
$$2j+(k-j)(-1) = j(-1)+2(k-j) = 0,$$
a contradiction since $k \geq 1$.

If $|\zent G| = 2$, then $D = 10$, and $\chi_i(h) = 1 \neq 0$ for an element $h \in G$ of order $3$, contradicting
Lemma \ref{key} applied to $h$.

Assume now that $|\zent G| = 3$. Then $D = 15$ and $H = P \times \zent G$ with $P \in \Syl_2(G)$. However, a faithful character in $\Irr(G)$ of degree $9$ contains $\reg_P$ on restriction to $P$ and so must occur in $\gamma^G$, a contradiction.

\smallskip
{\it Case 4:} $S = \PSL(3,4)$. Taking $g_5 \in G$ of order $5$, we see that $g_5$ fulfills the conditions of 
\ref{key}(ii) with $\al$ specified in \eqref{eq:qs3}. Taking $g_7 \in G$ of order $7$, we see that $g_7$ fulfills the conditions of 
\ref{key}(ii) with $\al$ specified in \eqref{eq:qs4}. Thus if $p \neq 5$ then by Lemma \ref{key} applied to $g_5$ we have that $\gamma^G(g_5) = 0$ and 
$\al+\al^*=0$. As $b5+b5^* = -1$, we conclude that $5|D$. Likewise, if $p \neq 7$ then Lemma \ref{key} applied to $g_7$ yields that $7|D$.

Now if $p \neq 5,7$, then we have that $35|D$; in particular, $|\zent G| \leq 2$. If $|\zent G| = 1$, then $D = 35$, and $\chi_i(g_2) = 3 \neq 0$ and 
$\chi_i(g_3) = -1 \neq 0$ for an element $g_2 \in G$ of order $2$ and an element $g_3 \in G$ of order $3$. This contradicts
Lemma \ref{key} applied to $g_3$ when $p \neq 3$ and to $g_2$ when $p \neq 2$. Likewise,  if $|\zent G| = 2$, then $D = 70$, and 
$\chi_i(g_2) = -2 = \chi_i(g_3)$ for an element $g_2 \in G$ of order $2$ and an element $g_3 \in G$ of order $3$, again a contradiction.

Assume now that $p = 5$ or $7$, whence $H = P \times \zent G$. In each of these cases, one can find two faithful characters in $\Irr(G)$ of distinct 
degrees that both contain $\reg_P$ on restriction to $P$, contradicting Lemma \ref{key}(iii).

\smallskip
{\it Case 5:} $S = \PSU(4,3)$. Taking $g_5 \in G$ of order $5$, we see that $g_5$ fulfills the conditions of 
\ref{key}(ii) with $\al$ specified in \eqref{eq:qs3}. Taking $g_7 \in G$ of order $7$, we see that $g_7$ fulfills the conditions of 
\ref{key}(ii) with $\al$ specified in \eqref{eq:qs4}. Thus if $p \neq 5$ then by Lemma \ref{key} applied to $g_5$ we have that $\gamma^G(g_5) = 0$ and 
$\al+\al^*=0$. As $b5+b5^* = -1$, we conclude that $5|D$. Likewise, if $p \neq 7$ then Lemma \ref{key} applied to $g_7$ yields that $7|D$.

Now if $p \neq 5,7$, then we have that $35|D$. In all of these cases, one of the following holds:
\begin{enumerate}
\item[(3.5.a)] One can find a $p'$-element $h \in G \smallsetminus \zent G$ and $\beta \neq 0$ such that $\chi_i(h) = \beta$ for all $\chi_i$ occurring in
\eqref{eq:qs2}. 
\item[(3.5.b)] One can find two $p'$-elements $h,h' \in G \smallsetminus \zent G$ and pairs $(\beta_1,\beta_1') \in \C^2$ and 
$(\beta_2,\beta_2') \in \C^2$ such that $(\chi_i(h),\chi_i(h')) = (\beta_1,\beta_1')$ or $(\beta_2,\beta_2')$ for all $\chi_i$ occurring in
\eqref{eq:qs2}. Furthermore, the system of equations $x_1\beta_1+x_2\beta_2=x_1\beta'_1+x_2\beta'_2 = 0$ have only one solution
$x_1=x_2=0$.
\end{enumerate}
Certainly, (3.5.a) contradicts Lemma \ref{key}(iii). In the case of 
(3.5.b), if we let $x_1$ be the number of $\chi_i$ in \eqref{eq:qs2} with $(\chi_i(h),\chi_i(h')) = (\beta_1,\beta_1')$ and $x_2$ the number of the 
remaining $\chi_i$, then by Lemma \ref{key}(i) we must have 
$$x_1\beta_1+x_2\beta_2=\gamma^G(h)=0=\gamma^G(h') = x_1\beta'_1+x_2\beta'_2,$$ 
whence $x_1+x_2=k=0$, again a contradiction. 

Assume now that $p = 5$ or $7$, whence $H = P \times \zent G$. In each of these cases, one can find two faithful characters in $\Irr(G)$ of distinct 
degrees that both contain $\reg_P$ on restriction to $P$ (in fact, they can be chosen to have $p$-defect $0$, unless $\zent G = 12_2$ 
in the notation of \cite{ATLAS}). This contradicts Lemma \ref{key}(iii).
\end{proof}

\begin{rem}
{\em Note that Theorem \ref{qs} does not hold when $(S,p) = (\mathrm{Alt}(5),5)$ and $(\mathrm{Alt}(6),3)$, even with $\gamma \in \Irr(H)$ assumed to be linear.
Indeed, taking $H = P \times \zent G$ with $P \in \Syl_p(G)$, $\gamma|_P = 1_P$, and $\gamma|_{\zent G}$ to be faithful, 
we have $\gamma^G = 2\chi$ for some irreducible character $\chi$ 
of degree $6$ of $G = 2\mathrm{Alt}(5)$ in the former case, and $\gamma^G = 2(\chi+\chi')$ for some irreducible characters $\chi$ and $\chi'$ 
of degree $10$ of $G = 2\mathrm{Alt}(6)$ in the latter case.
}
\end{rem}

\section{Induction and Central Products}
Suppose that the finite group $E$ is the central product of subgroups
$X_1, \ldots, X_n$.  By this we mean that $X_i \le E$, $[X_i,X_j]=1$
for $i \ne j$, $Z=\bigcap_{j=1}^n X_i$, and $E/Z = (X_1/Z) \times
\cdots \times (X_n/Z)$, that is, $(\prod_{j\ne i} X_j) \cap X_i=Z$ for
all $i$.  We fix $\lambda \in \mathrm{Irr}(Z)$.

Suppose that $\chi_i$ is a character of $X_i$ all of whose irreducible
constituents lie over $\lambda$. We claim that there is a unique
character $\chi_1 \cdot \ldots \cdot \chi_n$ of $E$, all of whose
irreducible constituents lie over $\lambda$, such that
\[
(\chi_1 \cdot \ldots \cdot \chi_n)(x_1 \cdots x_n)=
\chi_1(x_1) \cdots  \chi_n(x_n)
\]
for $x_i \in X_i$.

Let $E^*=X_1 \times \cdots \times X_n>Z\times \cdots\times Z$ and
let 
\[
A=\{(x_1, \ldots, x_n) \in Z \times \cdots \times Z |\ x_1
\cdots x_n=1\}.
\] 
Then $E^*/A$ is naturally isomorphic to $E$, via the
homomorphism $f$ given by $(x_1, \ldots, x_n) \mapsto x_1 \cdots x_n$
with kernel $A$.  Notice that $A$ is contained in the kernel of
$\chi=\chi_1 \times \cdots \times \chi_n$, and therefore $\chi$
naturally corresponds to a unique character $\psi$ of $E$ such that
$\psi(f(g))=\chi(g)$ for $g \in E^*$.  The character $\psi$ is what we
have called $\chi_1 \cdot \ldots \cdot \chi_n$.

Furthermore, by \cite[Theorem 10.7]{N2}, the map
\begin{equation}\label{eq:central_character}
  \mathrm{Irr}(X_1|\lambda) \times \cdots \times \mathrm{Irr}(X_n|\lambda) 
  \rightarrow \mathrm{Irr}(E|\lambda)
\end{equation}
given by $(\theta_1, \ldots, \theta_n) \mapsto \theta_1 \cdot \ldots
\cdot \theta_n$ is a bijection.

\begin{lem}\label{thislemma}
  Suppose now that we have a subgroup $K$ of $E$ of the form $K=K_1
  \cdots K_n$, where $Z \le K_i \le X_i$.  Suppose that $\gamma_i \in
  \mathrm{Irr}(K_i|\lambda)$. Then
  \[
  (\gamma_1 \cdot \ldots \cdot \gamma_n)^E=
  (\gamma_1)^{X_1} \cdot \ldots \cdot (\gamma_n)^{X_n} \, .
  \]
\end{lem}

\begin{proof}
  Again, let $E^*$, $A$ and $f:E^*\mapsto E$ as before.  Let $K^*=K_1
  \times \cdots \times K_n$, so $A<K^*$ and $K^*/A\simeq
  f(K^*)=K$. Since the map defined in \eqref{eq:central_character}
  commutes with the induction of characters, it is enough to check
  that
  \[
  (\gamma_1 \times \cdots \times \gamma_n)^{X_1 \times \cdots \times X_n}=
  \gamma_1^{X_1} \times \cdots \times \gamma_n^{X_n},
  \]
  but this is an easy exercise. 
\end{proof}

\section{Proof of Theorem A}

\medskip

In order to prove Theorem A, we shall use the following.

\begin{thm}\label{thm}
Suppose that $H$ is a nilpotent subgroup of $G$,
and $N \nor G$ is nilpotent. If $\gamma \in \irr H$ is such that
$\gamma^{HN} \in \irr{HN}$, then $HN$ is nilpotent.
\end{thm}

\begin{proof}
This is Corollary 2.3 of \cite{N1}.
\end{proof}

Next we prove our main result.

\begin{thm}\label{thethm}
Let $G$ be a finite group and let $H$ be a nilpotent subgroup
of $G$. Suppose that $\gamma \in \irr H$ is such that $\gamma^G \in \irr G$.
Then ${\bf F^*}(G)={\bf F}(G)$.
\end{thm} 

\begin{proof}
Let $G$ be a counterexample to Theorem \ref{thethm} such
that 
$|G|$ is as small as possible. Of course, if $H=G$, then $G$ is
nilpotent and there is nothing to prove.  We may assume therefore
that $H$ is a maximal subgroup of $G$ with respect to being
nilpotent.

Since ${\bf F}^{*}(G)$ strictly contains ${\bf F}(G)$, the group $G$
is not only non-solvable but it contains a subnormal quasi-simple
group, say $X$. Let $E$ be the product of all $H$-conjugates of $X$ in
$G$. Since $G$ is a minimal counterexample, we have $G = HE$.

The group $E$ is the layer of $G$ and is a perfect central extension
of a direct product of simple groups which are transitively permuted
by $H$. Thus $E/\zent E$ is a perfect minimal normal subgroup of
$G/\zent E$. Write $E/\zent E$ as $E/\zent E = S_{1} \times \cdots
\times S_{n}$ for some pairwise isomorphic simple groups $S_{i}$ with
$i$ in $\{ 1, \ldots, n \}$ and integer $n$. Put $\chi =\gamma^{G}$.

\medskip
  \Step~~{\sl The irreducible character $\chi$ must be faithful.}
  \label{step:Faithful}
\medskip 

The kernel of $\chi$ is equal to $U = \cap_{x \in G}{(\mathrm{ker}(\gamma))}^{x}$ by \cite[Lemma 5.11]{Isaacs} which is a  
normal subgroup of $G$ contained in $H$.
Thus $U$ is nilpotent.  Now $\gamma$ may be viewed as a character of $H/U$ 
which induces  an irreducible character of $G/U$.
Since $QU/U$ is a component of
$G/U$, it easily follows that $G/U$ is a counterexample to the conjecture 
with order at most $|G|$. This can only happen if $U = 1$.

\medskip
  \Step~~{\sl If $A \nor G$ is nilpotent, then $A \le H$ and
  $\chi_A$ is a multiple of some $\tau \in \irr A$. In particular, if
  $\tau(1)=1$, then     
  $A \sbs \zent G$ is cyclic.  Hence, every normal abelian subgroup
  of $G$ is cyclic and central.}\label{step:NilpNormal}

\medskip 
By Theorem \ref{thm}, we have that $A \le H$ by using that
$H$ is a maximal subgroup of $G$ subject to being nilpotent.  Let
$\tau \in \irr A$ be an irreducible constituent of $\gamma_A$.  Hence
$\tau$ also lies under $\chi$. Let $T$ be the stabilizer of $\tau$ in
$G$.  Since $[E, {\bf F}(G)]=1$ and $A \leq {\bf F}(G)$, we have that
$E \sbs T$.  Now, if $\epsilon \in \irr{T\cap H|\tau}$ is the Clifford
correspondent of $\gamma$ over $\tau$, it follows that
$\epsilon^G=\gamma^G=\chi$ is irreducible. Hence, $\epsilon^T$ is also
irreducible.  Since $X$ is a component of $T$, we would have a
contradiction in case $T<G$.  Thus $T=G$, which means that $\chi_A$ is
a multiple of $\tau$ by Clifford theory. 
Since $\chi$ is faithful by Step \ref{step:Faithful}, so is $\tau$. 
If $\tau$ is linear, then $A \le \zent G$ since $\tau$ and $A$ are 
$G$-invariant. It also follows that $A$ is cyclic. 


\medskip
  \Step~~{\sl We have $\Fit =C_G(E/\zent E)$
  and $\Fit \cap E=\zent E$.}\label{step:Sol=F=CE}
\medskip 

Let $M$ be the largest normal solvable subgroup of $G$. 
Since $\zent E \leq M$ and $E/\zent E$ is a
direct product of non-abelian simple groups, it follows that $M \cap E
= \zent E$. The group $M/\zent E$ is isomorphic to $ME/E$ which is
nilpotent (because it is isomorphic to a subgroup of $H/(H\cap
E)$). Since $\zent E$ is contained in $\zent G$ by Step
\ref{step:NilpNormal}, it follows that $M$ is nilpotent.  Now, it is
clear that $M \leq \Fit\leq C_G(E)$.
Since $C_G(E)/\zent E$ is isomorphic to a subgroup of the nilpotent
group $G/E$, so $C_G(E)$ is a normal solvable subgroup of $G$, which
proves that $C_G(E)\leq M$. Thus, $M=\Fit=C_G(E)$ follows. Finally, 
$C_G(E)\leq C_G(E/\zent E)$ is clear, while
if $g\in C_G(E/\zent E)$, then $[x,y]^g=[x^g,y^g]=[x,y]$ holds for every 
$x,y\in E$, so $g\in C_G(E)$ by using that $E$ is perfect. 

\medskip
  \Step~~{\sl $|H/\Fit|^2\geq |G/\Fit|$.}\label{step:HperF2-GperF}
\medskip 

Write $F=\Fit$. By Step \ref{step:NilpNormal}, we know that $\chi_F$
is a multiple of $\tau \in \irr F$. Now we use the theory of character
triples, as developed in Section 5.4 of \cite{N2}, and with the same
notation.  Let $(G^*,F^*,\tau^*)$ be a character triple isomorphic to
$(G,F, \tau)$, where $F^*$ is central, and $\tau^*$ is faithful. Let
$(H/F)^*=H^*/F^*$.  If $\gamma^* \in \irr{H^*|\tau^*}$ corresponds to
$\gamma$, then we have that $(\gamma^*)^{G^*}=\chi^*$ is
irreducible. (See the discussion before Lemma 5.8 of \cite{N2}.)
Using that $|G^*:F^*|=|G:F|$ and $|H^*:F^*|=|H:F|$ we get that
\[
|G:F|=|G^*:F^*|\geq \chi^*(1)^2=\gamma^*(1)^2\cdot |G^*:H^*|^2= 
\gamma^*(1)^2 \cdot \frac{|G^*:F^*|^2}{|H^*:F^*|^2}
\geq \frac{|G:F|^2}{|H:F|^2},
\]
where the first inequality follows from \cite[Lemma 2.27(f) and
Corollary 2.30]{Isaacs} and from the fact that $F^*$ is central in
$G^*$.

\medskip
  \Step~~{\sl We have that $n > 1$ and $E\cap H\neq \zent E$.}
  \label{step:n>1,L1>1}
\medskip

If $n = 1$, then $G/\Fit$ is almost-simple and
$|H/\Fit|^2 < |G/\Fit|$ by Theorem C contradicting Step \ref{step:HperF2-GperF}.

Suppose now that $E\cap H=\zent E$, so $H\cap (\Fit E)=\Fit$ and 
$E/Z(E)\cong \Fit E/\Fit$. 
By Step \ref{step:Sol=F=CE}, the kernel of the action of $H$ on
$E/Z(E)\simeq S_1\times\ldots\times S_n$ equals $\Fit$, 
so this induces an inclusion of the nilpotent group
$H/\Fit$ into $W = \mathrm{Out}(S_{1}) \wr \mathrm{Sym}(n)$. If $\psi$
denotes the natural map from $W$ to $W/{(\mathrm{Out}(S_{1}))}^{n}$,
then $\psi(H/\Fit)$ may be viewed as a nilpotent subgroup of
$\mathrm{Sym}(n)$. We have $|\psi(H/\Fit)| \leq 2^{n-1}$ by
\cite[Theorem 3]{Dixon}. Thus $|H/\Fit| \leq |{\rm Out}(S_1)|^{n}
\cdot 2^{n-1}$. By a remark after \cite[Lemma 7.7]{GMP}, it follows
that $|{\rm Out}(S_1)| < |S_1|^{1/2}/2$. We conclude that
\[|H/\Fit|^2 < |\mathrm{Out}(S_1)|^{2n} \cdot 2^{2n} \leq |S_1|^n
=|\Fit E/\Fit|\leq |G/\Fit|.\] Again, this contradicts the bound
$|H/\Fit|^2 \geq |G/\Fit|$ obtained in Step \ref{step:HperF2-GperF},
thus $n > 1$ and $E\cap H\neq \zent E$ as claimed.

\medskip
\Step~~{\sl We have $(E \cap H)/\zent E = L_{1} \times \cdots \times
  L_{n}$ for some pairwise isomorphic nilpotent subgroups $L_{i} <
  S_{i}$ with $i$ in $\{ 1, \ldots , n \}$, which are transitively
  permuted by $H$.\label{step:L1xxLn} } 
\medskip

Let $L_{i}$ be the projection of $(E \cap H)/\zent E$ to $S_{i}$ for
$1 \leq i \leq n$.  Since $H/\zent E$ acts transitively on the set $\{
S_{1}, \ldots, S_{n} \}$ by conjugation, it also acts transitively on
the set $\{ L_{1}, \ldots, L_{n} \}$ by conjugation. Thus the groups
$L_{i}$ are pairwise isomorphic and nilpotent for $i$ in $\{ 1, \ldots
, n \}$.

Let $K$ be the preimage of $L_{1} \times \cdots \times L_{n}$ in
$E$. Clearly, $K$ is a nilpotent subgroup of $E$. We claim that $K = E
\cap H$. By Theorem \ref{thm} and the maximality of $H$, it is
sufficient to show that $K$ is normalized by $H$. For this it is
sufficient to see that the preimage $K_1$ of $L_1$ in
$E$ has the property that $K_1^{h} \leq K$ for every $h \in H$. 
But this is clear.

\medskip

The group $H/\zent E$ acts by conjugation on $E/\zent E$. As noted in
the proof of Step \ref{step:L1xxLn}, this action induces a transitive
action on the set $\Omega = \{ S_{1}, \ldots , S_{n} \}$ of simple
direct factors of $E/\zent E$. Let $B/\zent E$ be the kernel of this
action. Thus $B$ is a normal subgroup of $H$ containing $\zent E$ with
the property that $H/B$ can be considered as a transitive subgroup of
$\mathrm{Sym}(\Omega)$.

\medskip
  \Step~~{\sl Both $(E \cap H)/\zent E$ and $H/B$ are $p$-groups for the 
    same prime $p$.}\label{step:EcapH,HperB-p}
\medskip

By Step \ref{step:n>1,L1>1} we know that $n > 1$ and $L_{1} \not= 1$. Consider $H/B$ as a transitive subgroup of $\mathrm{Sym}(\Omega)$. By our assumptions, $H/B$ is a non-trivial nilpotent group. Let $p$ be a prime divisor of $|H/B|$. Since no non-trivial normal subgroup of $H/B$ can stabilize $S_{1} \in \Omega$, the Sylow subgroup ${\bf O}_{p}(H/B)$ of $H/B$ cannot stabilize $S_{1}$. It follows that neither the Sylow subgroup ${\bf O}_{p}(H / \zent E)$ of $H/\zent E$ can stabilize $S_{1}$. Let $x$ be a $p$-element in $H/\zent E$ which does not stabilize $S_{1}$. Then $x$ cannot stabilize $L_{1} \leq (E \cap H)/\zent E$ either. By Step \ref{step:L1xxLn} and its proof, we know that $L_{1} \cap {(L_{1})}^{x} = 1$. Since $H/\zent E$ is nilpotent, this can only occur if $L_{1}$ is a $p$-group. Since $L_{1}$ is a non-trivial $p$-group where $p$ is an arbitrary prime divisor of $|H/B|$, we conclude that both $L_{1}$ and $H/B$ are (non-trivial) $p$-groups for the same prime $p$. The result now follows by Step \ref{step:L1xxLn}.  

\medskip

Let $T$ denote the preimage in $G$ of the kernel of the action of $G/\zent E$ on $\Omega$. So $H \cap T = B$ and $T = BE$. 
Recall that $\mathrm{LIST} = \{ \mathrm{Alt}(5), \mathrm{Alt}(6),
\mathrm{PSL}(2,7), \mathrm{PSL}(3,4), \mathrm{PSU}(4,3) \}$.

Define $c=1$ if $S_1 \in \mathrm{LIST}$, $L_{1}$ is a Sylow $p$-subgroup of $S_1$ and $p =2$ in case $S_{1} = \mathrm{Alt}(5)$ or $\mathrm{Alt}(6)$.
Otherwise, define $c=2$.  

\medskip
  \Step~~{\sl ${|B/\Fit|}^{2} \leq c^{-n} \cdot |T/\Fit|$.}
  \label{step:BperF2-TperF} 
\medskip

By Step \ref{step:Sol=F=CE}, the group $T/\Fit$ is isomorphic to a
subgroup $\overline{T}$ of $\mathrm{Aut}(S_{1}) \times \cdots
\times \mathrm{Aut}(S_{n})$ containing the normal subgroup $S_{1}
\times \cdots \times S_{n}$ and the group $B/\Fit$ may be viewed as a
nilpotent subgroup $\overline{B}$ of $\overline{T}$. We show the claim
by induction on $n$. First, for $n = 1$ the claim follows from Theorem
C if $c=1$, while it follows from Theorem \ref{almostsimple} if
$c=2$. Now, let us assume that $n > 1$ and that the claim is
true for $n-1$. Let $\pi$ be the natural projection of $\overline{T}$ to
$\mathrm{Aut}(S_{1})$. Then $c \cdot {|\pi(\overline{B})|}^{2} \leq
|\pi(\overline{T})|$. Moreover, ${|\mathrm{ker}(\pi) \cap
  \overline{B}|}^{2} \leq c^{-n+1} \cdot |\mathrm{ker}(\pi) \cap
\overline{T}|$ by the fact that the claim is true for
$n-1$. Thus 
\begin{align*}
|B/\Fit|^2 &= |\overline{B}|^2 = |\pi(\overline{B})|^2 \cdot 
            |\mathrm{ker}(\pi) \cap \overline{B}|^2 \leq\\
           &\leq c^{-n} \cdot |\pi(\overline{T})| \cdot |\mathrm{ker}(\pi) \cap 
           \overline{T}| = c^{-n}\cdot|\overline{T}|=c^{-n}\cdot |T/\Fit|.
\end{align*}
   
\medskip
  \Step~~{\sl $S_{1} \in \mathrm{LIST}$ and $L_{1}$ is a Sylow $p$-subgroup of $S_1$. Moreover, $p=2$ in case $S_1$ is $\mathrm{Alt}(5)$ or $\mathrm{Alt}(6)$.} 
\medskip

By Steps \ref{step:n>1,L1>1} and \ref{step:EcapH,HperB-p}, we know that $n > 1$ and $L_1 \not= 1$ is
a $p$-group for some prime $p$. By the definition of $c$, we need
to prove that $c=1$.
We have $|B/\Fit|^2 \leq c^{-n} \cdot |T/\Fit|$ by Step \ref{step:BperF2-TperF}. Thus
\begin{equation}
\label{e1}
|H/\Fit|^2 = |H/B|^2 \cdot |B/\Fit|^2 \leq c^{-n} \cdot |H/B|^2 \cdot |T/\Fit|.
\end{equation}
Since $G=HT$ and $H\cap T=B$, we have 
\begin{equation}
\label{e2}
|H/B| \cdot |T/\Fit| = \frac{|H||T|}{|B||\Fit|} = \frac{|HT|}{|\Fit|} = |G/\Fit|.
\end{equation}
Inequalities (\ref{e1}) and (\ref{e2}) give 
\begin{equation}
\label{e3}
|H/\Fit|^2 \leq c^{-n}\cdot|H/B|\cdot|G/\Fit| \leq c^{-n}\cdot 2^{n-1} \cdot |G/\Fit|
\end{equation}
where the second bound follows from Step \ref{step:EcapH,HperB-p}, noting that $H/B$ is a $p$-subgroup of the symmetric group on $n$ letters and thus it has size at most $2^{n-1}$.
(It is a well-known fact that the $p$-part of $n!$
is at most this number.) Now Step \ref{step:HperF2-GperF} and (\ref{e3}) give
\[
|G/\Fit| \leq |H/\Fit|^2 \leq  c^{-n} \cdot 2^{n-1} \cdot |G/\Fit|
\]
from which $1 \leq c^{-n} \cdot 2^{n-1}$ follows, forcing $c=1$.

\medskip
  \Step~~{\sl The character $\gamma_{H \cap E}$ is irreducible.}\label{step:GammaRestriction}
\medskip

Let $U$ be any proper subgroup of $H$ containing $H \cap E$. We claim that $\mu^{H} \not= \gamma$ for every irreducible character $\mu$ of $U$. For a contradiction assume that $\mu^{H} = \gamma$ for some $\mu \in \mathrm{Irr}(U)$. Since $\gamma^{G}$ is irreducible, $\mu^{EU}$ is also irreducible. 
Since $|EU:U|=|G:H|$
and $|EU| < |G|$, by induction
we will have that ${\bf F^*}(EU)$ is nilpotent, but this
cannot happen.   

Let $H = U_{0} > \ldots > U_{t} = H \cap E$ be a chain of normal subgroup of $H$ with $t \geq 1$ maximal. By repeated application of \cite[Theorem 6.18]{Isaacs} and the claim in the previous paragraph, the character $\gamma_{U_{i}}$ is homogeneous for every index $i$ with $0 \leq i \leq t$. Moreover, since $H$ is nilpotent and $t$ is maximal, $|U_{i}/U_{i+1}|$ is prime for every index $i$ with $0 \leq i \leq t-1$, and so $\gamma_{U_{i}}$ is irreducible for every index $i$ with $0 \leq i \leq t$. In particular $\gamma_{H \cap E}$ is irreducible. 

\medskip

As in the proof of Step \ref{step:L1xxLn}, for every $i$ with $1 \leq i \leq n$, let $K_i \leq H \cap E$ be the preimage of $L_{i} < S_{i} \leq E/\zent E$ in $E$. In particular $L_{i} \cong K_i / \zent E$.

\medskip
  \Step~~{\sl $H \cap E$ is a central product of the subgroups $K_1, \ldots , 
    K_n$ amalgamating $\zent E$.}\label{step:L1Ln-central} 
\medskip

For each $i$ with $1 \leq i \leq n$, the group $K_i$ contains $\zent
E$. By Step \ref{step:L1xxLn}, $H \cap E = \langle K_1, \ldots , K_n
\rangle$. Since every distinct pair of components of $G$ commute, we
have $[ K_i, K_j ] = 1$ for every $i$ and $j$ with $1 \leq i < j \leq
n$. Finally, for every index $i$ with $1 \leq i \leq n$, the
intersection of $K_i$ with $\langle K_1, \ldots , K_{i-1}, K_{i+1},
\ldots K_n \rangle$ is $\zent E$.

\medskip
 {\sl Final Step.} 
\medskip

By Step \ref{step:GammaRestriction}, $\gamma_{H\cap E}\in\irr{H\cap E}$, 
so $\gamma_{H\cap E}=\gamma_1 \cdot \ldots \cdot \gamma_n$ for
$\gamma_i \in \irr{K_i|\lambda}$, where $\chi_{\zent
  E}=\chi(1)\lambda$ by Step \ref{step:L1Ln-central} and by Section 4.

For every $i$ with $1\leq i\leq n$, let $X_i$ be the preimage of 
$S_i\leq E/\zent E$ in $E$, so $E$ is the central product of $X_1,\ldots,X_n$. 
By Mackey and Clifford's theorem, we have that all irreducible constituents of
\[
\chi_{E}=(\gamma_{H\cap E})^{E}=(\gamma_1 \cdot \ldots \cdot \gamma_n)^{E}
\] 
have equal degrees. By Lemma \ref{thislemma}, this character equals
\[
(\gamma_1)^{X_1} \cdot \ldots \cdot (\gamma_n)^{X_n} \, .
\]
For $k>1$, fix an irreducible constituent $\rho_k$ of 
$(\gamma_k)^{X_k}$. By Theorem \ref{qs}, let $\xi_1$ and $\xi_2$  be irreducible
constituents of $(\gamma_1)^{X_1}$ with different degrees. Then 
$\xi_i \cdot \rho_2 \cdot \ldots \cdot \rho_n$ with $i \in \{ 1,2 \}$
are two irreducible constituents of $\chi_{E}$ with different degrees.
This contradiction proves the theorem.
\end{proof}

Notice that Theorem A easily follows from Theorem \ref{thm} and Theorem 
\ref{thethm}.

\end{document}